\documentclass[letterpaper,12pt]{amsart}

\usepackage{amsmath,amssymb,amsthm,commath}
\usepackage{enumitem,graphicx}

\theoremstyle{plain}
\newtheorem{theorem}{Theorem}
\newtheorem{lemma}[theorem]{Lemma}
\newtheorem{corollary}[theorem]{Corollary}
\newtheorem*{mainthm}{Main Theorem}

\newcommand{\ithcenter}[2]{\mathbf{Z}_{#1}(#2)}
\newcommand{\centerof}[1]{\mathbf{Z}(#1)}

\begin{document}
	
\title[Commuting Graphs]{Extending Morgan and Parker's results about commuting graphs}

\author[Beike]{Nicolas F. \ Beike}
\address{Department of Mathematics and Statistics, 501 Lincoln Building, Youngstown State University, Youngstown, OH 44555}
\email{nfbeike@student.ysu.edu}

\author[Carleton]{Rachel Carleton}
\address{Department of Mathematical Sciences, Kent State University, Kent, OH 44242}
\email{rcarlet3@kent.edu}

\author[Costanzo]{David G.\ Costanzo}
\address{School of Mathematical and Statistical Sciences, O-110 Martin Hall, Box 340975, Clemson University, Clemson, SC 29634}
\email{dcostan2@kent.edu}

\author[Heath]{Colin Heath}
\address{Department of Mathematics, 3620 S. Vermont Ave., KAP 104, University of Southern California, Los Angeles, CA 90089}
\email{colinjhe@usc.edu}

\author[Lewis]{Mark L.\ Lewis}
\address{Department of Mathematical Sciences, Kent State University, Kent, OH 44242}
\email{lewis@math.kent.edu}

\author[Lu]{Kaiwen Lu}
\address{Department of Mathematics, 2074 East Hall, 530 Church Street, University of Michigan, Ann Arbor, MI 48109}
\email{kailu@umich.edu}

\author[Pearce]{Jamie D. Pearce}
\address{Department of Mathematics, University of Texas at Austin, 2515 Speedway, PMA 8.100, Austin, TX 78712}
\email{jamie.pearce@utexas.edu}


\begin{abstract}
Morgan and Parker have proved that if $G$ is a group satisfying the condition that $\centerof{G} = 1$, then the connected components of the commuting graph of $G$ have diameter at most $10$. Parker has proved that if in addition $G$ is solvable, then the commuting graph of $G$ is disconnected if and only if $G$ is a Frobenius group or a $2$-Frobenius group, and if the commuting graph of $G$ is connected, then its diameter is at most $8$.  We prove that the hypothesis $Z (G) = 1$ in these results can be replaced with $G' \cap \centerof{G} = 1$.  We also prove that if $G$ is solvable and $G/\centerof{G}$ is either a Frobenius group or a $2$-Frobenius group, then the commuting graph of $G$ is disconnected. 	
\end{abstract}

\maketitle
	
\section{Introduction}
	
In this paper, all groups are finite.  Given a group $G$, the commuting graph of $G$, denoted by $\Gamma(G)$, is defined as the graph whose vertex set is the noncentral elements of $G$ and two vertices are adjacent if and only if they commute in $G$.  Giudici and Pope in \cite{GiuPo} state that commuting graphs were first studied by Brauer and Fowler in \cite{BrFo} in relation to the classification of simple groups.  They go on to say that commuting graphs were first studied in their own right by Segev and Seitz in \cite{SeSe} in terms of the classical simple groups.  In fact, much of the research regarding the commuting graph is related to simple groups.  We are not going to try describe all of this research, but note that it culminates in the work of Solomon and Woldar in \cite{SoWo} were they complete the final step in proving that if $S$ is a simple group and $X$ is any group where $\Gamma (S) \cong \Gamma (X)$, then $S \cong X$.

In the course of this research, Iranmanesh and Jafarzedah conjecture in \cite{IraJaf} that there is a universal bound on the diameter of commuting graphs.   Giudici and Pope in \cite{GiuPo} study the diameters of commuting graphs for a number of families of group.  Surprisingly, Giudici and Parker produced in \cite{GiuPark} a family of $2$-groups of nilpotence class $2$ where there is no bound on the diameter of the commuting graphs.  On the other hand, in the seminal 2013 paper \cite{parkerSoluble}, Parker proves that the commuting graph of a solvable group $G$ with a trivial center is disconnected if and only if $G$ is a Frobenius or $2$-Frobenius group.  We will state the definition of $2$-Frobenius groups in Section 2.  In addition, when the commuting graph of $G$ is connected, he proves that it has diameter at most $8$.  He also provides an example of a solvable group where this bound is met. Parker and Morgan remove the solvability hypothesis on $G$ in \cite{parkerTrivialCenter} by proving for the commuting graph of any group with a trivial center that all the connected components have  diameter at most $10$.      

In this paper, we look to extend Parker's and Parker and Morgan's results.  In particular, we would like to extend the class of groups where we can determine which commuting graphs are disconnected and determine further classes of groups where we can show that diameters of the commuting graph are bounded by the constants $8$ and $10$.  In particular, we show that we can replace the hypothesis that $\centerof{G} = 1$ with the hypothesis that $G' \cap \centerof{G} = 1$.  

\begin{mainthm}
Let $G$ be a group and suppose that $G' \cap \centerof{G} = 1$.  
\begin{enumerate}
\item $\Gamma (G)$ is connected if and only if $\Gamma (G/\centerof{G})$ is connected.
\item Every connected component of $\Gamma (G)$ has diameter at most $10$.
\item If $G$ is solvable and $\Gamma (G)$ is connected, then $\Gamma (G)$ has diameter at most $8$.
\item If $G$ is solvable, then $\Gamma (G)$ is disconnected if and only if $G/Z$ is either a Frobenius group or a $2$-Frobenius group.
\end{enumerate}
\end{mainthm}

In fact, we will see that we can relax the hypothesis that $G' \cap \centerof{G} = 1$ even further.  In particular, we show that it suffices to assume that $C (G) \cap \centerof{G} = \{ 1 \}$ where $C (G) = \{ [x,y] \mid  x,y \in G \}$.  I.e., we only need the set of commutators, not the whole commutator subgroup.  (We note that we use $1$ to denote the trivial subgroup.  Since $C(G)$ is not a subgroup, it is not appropriate to view $C(G) \cap \centerof{G}$ as a subgroup, so we use $\{ 1 \}$ to denote the set consisting only of the identity.)  We will present examples of groups $G$ where $C (G) \cap \centerof{G} = \{ 1 \}$, but $G' \cap Z (G) > 1$, so this replacement does actually improve the result.

Following the literature, we say that a group $G$ is an {\it $A$-group} if every Sylow subgroup of $G$ is abelian.  We show that if $G$ is an $A$-group then $G' \cap \centerof{G} = 1$.  In particular, this shows that the results of the Main Theorem apply to $A$-groups; however, we would not be surprised if one could show that the diameter bounds can be lowered for $A$-groups, especially solvable $A$-groups, but we have not investigated this at this time.

One other consequence of $G' \cap \centerof{G} = 1$ is that $\centerof{G/\centerof{G}} = 1$, and so it makes sense to ask whether we can determine which groups satisfying $\centerof{G/\centerof{G}} = 1$ have a disconnected commuting graph and if we can bound the diameter of the commuting graphs of groups where the commuting graph is connected.  At this point, we can find one class of groups where the commuting graph is disconnected as seen in the following theorem, but we will present examples of other groups satisfying this condition where the commuting graph is disconnected.  

\begin{theorem}
If $G$ is a group where $G/\centerof{G}$ is either a Frobenius or a $2$-Frobenius group, then $\Gamma (G)$ is disconnected.
\end{theorem}

This research was conducted during a summer REU in 2020 at Kent State University with the funding of NSF Grant DMS-1653002.
We thank the NSF and Professor Soprunova for their support.
	

\section{Results}	
	
Let $G$ be a finite group and $\Gamma(G)$ the commuting graph of $G$.  When $x = y$ or $x$ and $y$ are adjacent in $\Gamma(G)$, we write $x \sim y$.  In particular, writing $x \sim y$ emphasizes that $x,y \in G \setminus \centerof{G}$.  Viewing $x,y$ as vertices of $\Gamma (G)$, we use $d(x,y)$ to denote the distance between $x$ and $y$. (I.e., the number of edges in the shortest path between $x$ and $y$.) For the remainder of the paper, we set $Z = \centerof{G}$, the center of the group, and let $Z_2 = \ithcenter{2}{G}$, the preimage of $\centerof{G/\centerof{G}}$. We also set $C = C(G)$.  Note that $G'$ is the group generated by $C$, but $C$ itself is not always a group.

	



The following lemma addresses the relationship between adjacent elements in $\Gamma(G)$ and adjacent elements in $\Gamma(G/Z)$.

\begin{lemma} \label{three}
Let $G$ be a group and fix $x, y\in G \setminus Z_2$.  If $x \sim y$ in $\Gamma(G)$, then $xZ \sim yZ$ in $\Gamma(G/Z)$.  If $C \cap Z = \{ 1 \} $ and $xZ \sim yZ$ in $\Gamma(G/Z)$, then $x \sim y$ in $\Gamma(G)$.
\end{lemma}

\begin{proof}
Suppose that $x \sim y$ in $\Gamma(G)$.  Then, we have that $xy=yx$, and hence, $xZ yZ = yZ xZ$ in $G/Z$.  Thus, $xZ \sim yZ$. 
	
Now, suppose $C \cap Z = \{ 1 \}$ and $xZ \sim yZ$.  This implies that $[x,y]\in Z$.  Also, $[x,y] \in C$, and since $C \cap Z = \{ 1 \}$, we obtain  $[x,y] = 1$.  Therefore, $xy = yx$ and so $x \sim y$.
\end{proof}

We now consider the relationship between $\Gamma (G)$ and $\Gamma (G/Z)$ when $Z_2 = Z$ and in particular, when $C \cap Z = \{ 1 \}$.

\begin{lemma} \label {four}
Let $G$ be a group and suppose that $Z_2  = Z$.  
\begin{enumerate}
\item If $\Gamma (G)$ is connected, then $\Gamma (G/Z)$ is connected and the diameter of $\Gamma (G/Z)$ is less than or equal to the diameter of $\Gamma (G)$.  
\item If $C \cap Z = \{ 1 \}$ and $\Gamma (G/Z)$ is connected, then $\Gamma (G)$ is connected and has the same diameter as $\Gamma (G/Z)$.
\item If $C \cap Z = \{ 1 \}$ and $\Gamma (G/Z)$ is disconnected, then there is a one-to-one correspondence between the connected components of $\Gamma (G/Z)$ and of $\Gamma (G)$ that preserves diameter except in the case that a connected component of $\Gamma (G/Z)$ consists of a single coset and $Z > 1$; in this case, the corresponding component of $\Gamma (G)$ will have diameter $1$.
\end{enumerate}
\end{lemma}

\begin{proof}
Suppose $x, y \in G \setminus Z$.  Since $\Gamma (G)$ is connected, we can find $x = x_0, x_1, \dots, x_n = y \in G \setminus Z$ so that $x_i \sim x_{i+1}$ for $i = 0, \dots, n -1$.  By Lemma \ref{three}, we have $x_i Z \sim x_{i+1}Z$.  It follows that $xZ$ and $yZ$ are connected by a path of length of $n$.  It follows that $d(xZ,yZ) \le d(x,y)$.  We deduce that $\Gamma (G/Z)$ is connected and its diameter is at most the diameter of $\Gamma (G)$.

We now add the assumption that $C \cap Z = \{ 1 \}$.  From Lemma \ref{three}, we see that $x \sim y$ in $\Gamma (G)$ if and only if $xZ \sim yZ$ in $\Gamma (G/Z)$.  It follows that $x$ and $y$ lie in the same connected component of $\Gamma (G)$ if and only if $xZ$ and $yZ$ lie in the same connected component of $\Gamma (G/Z)$.  In the case where they are in the same connected component, then $d (x,y) = d(xZ, yZ)$; except when $xZ = yZ$ and $x \ne y$, in this case $d(x,y) = 1$ and $d (xZ,yZ) = 0$.  Notice that if $xZ \ne yZ$, then $x \ne y$, so a connected component of $\Gamma (G/Z)$ with more than one element will correspond to a connected component of $\Gamma (G)$ with more than one element and they will have the same diameter.  Finally, if $\{ xZ \}$ is a connected component of $\Gamma (G/Z)$, then $\{ xz \mid z \in Z\}$ will be the corresponding connected component of $\Gamma (G)$, so the component in $\Gamma (G/Z)$ has diameter $0$, but the corresponding component in $\Gamma (G)$ will have diameter $1$.
\end{proof}

Following the literature, a group $G$ is a {\it Frobenius group} if it contains a proper, nontrivial subgroup $H$ so that $H \cap H^g = 1$ for all $g \in G \setminus H$.  The subgroup $H$ is called a {\it Frobenius complement} for $H$.  Frobenius proved that $N = (G \setminus \cup_{g \in G} H^g) \cup \{ 1 \}$ is a normal Hall subgroup of $G$ that satisfies $G = HN$ and $H \cap N = 1$; the subgroup $N$ is called the {\it Frobenius kernel} of $G$.  We say $G$ is a {\it $2$-Frobenius group} if there exist normal subgroups $K \le L$ so that $L$ and $G/K$ are Frobenius groups with Frobenius kernels $K$ and $L/K$ respectively.

We now show that when $G/Z$ is either a Frobenius group or a $2$-Frobenius group, then $\Gamma (G)$ is disconnected.

\begin{corollary}
If $G$ is a group so that $G/Z$ is either a Frobenius group or a $2$-Frobenius group, then $\Gamma (G)$ is disconnected.
\end{corollary}

\begin{proof}
We know that the center of a Frobenius group or a $2$-Frobenius group is trivial, so $Z_2 (G) = Z$.  Now, Parker has shown that $\Gamma (G/Z)$ is disconnected.  We apply the contrapositive of Lemma \ref{four}(1) to see that $\Gamma (G)$ is disconnected.
\end{proof}

We next show that the condition that $C \cap Z = \{ 1 \}$ implies that $Z = Z_2$.
For $x \in G$, set $D_G (x) = \{ d \in G \mid [d, x] \in Z \}$. Thus, $D_G(x)/Z = C_{G/Z} (Zx)$.

\begin{lemma}\label{five} 
If $G$ is a group such that $C \cap Z = \{ 1 \}$, then $C_G(x) = D_G(x)$ for each
$x \in G$. In particular, $Z = Z_2$.
\end{lemma}

\begin{proof} 
Let $x \in G$. Of course, $C_G(x) \leq D_G(x)$. If $d \in D_G(x)$, then $[d,x] \in
C \cap Z = \{ 1 \}$ and so $d \in C_G(x)$. Thus, $D_G(x) \le C_G(x)$, and we conclude that
$C_G (x) = D_G (x)$. Finally, observe that $Z = \cap_{x \in G} C_G(x) = \cap_{x \in G} 
D_G(x) = Z_2$, as wanted.
\end{proof}

 
 
We now obtain Parker and Morgan's and Parker's result for $G$ when $C \cap Z = \{ 1 \}$.
 
\begin{corollary} \label{six}
Let $G$ be a group and suppose that $C \cap Z = \{ 1 \}$.  
\begin{enumerate}
\item $\Gamma (G)$ is connected if and only if $\Gamma (G/Z)$ is connected.
\item Every connected component of $\Gamma (G)$ has diameter at most $10$.
\item If $G$ is solvable and $\Gamma (G)$ is connected, then $\Gamma (G)$ has diameter at most $8$.
\item If $G$ is solvable, then $\Gamma (G)$ is disconnected if and only if $G/Z$ is either a Frobenius group or a $2$-Frobenius group.
\end{enumerate}	
\end{corollary}

\begin{proof}
Conclusion (1) is an immediate consequence of Lemma \ref{four} (1) and (2) combined with Lemma \ref{five}.  Since $Z(G/Z) = 1$, Parker's results apply to $G/Z$, and applying Parker's results in $G/Z$ with Lemma \ref{four} (2) and (3) yields (2), (3), and (4).
\end{proof}


The following result has appeared in the literature. See, for example, Corollary 4.5 in \cite{Bro} or Theorem 4.1 in \cite{Tau }.

\begin{lemma} \label{agroup}
Let $G$ be an $A$-group.  Then $G' \cap Z =  1$.
\end{lemma}


Using Lemma \ref{agroup}, we see that if $G$ is an $A$-group, then $G$ satisfies the hypothesis of Corollary \ref{six}, and thus, $G$ satisfies the conclusions of Corollary \ref{six}.

We close by presenting some examples using the small groups library \cite{small} that can be accessed by the computer algebra systems GAP \cite{gap} or Magma \cite{magma} to illustrate various points.  We first present examples of groups $G$ where $C \cap Z = \{ 1 \}$ but $G' \cap Z > 1$.  Take $G$ to be one of SmallGroup(768,1083474), SmallGroup(768,1083475), or SmallGroup(768,1083476).  

We next present examples of groups $G$ where $\Gamma (G/Z)$ is connected but $\Gamma (G)$ is not connected.  Taking $G$ to be one of the Small Groups (72,22), (72,23), (120,11), (144,125), and (288,565) yield examples of groups $G$ with this property.  We actually have many more examples in this category, but we have just pulled a few examples at random from the list.  We include the graphs $\Gamma (G)$ and $\Gamma (G/Z)$ for $G = {\rm SmallGroup} (72,22)$ in Figures \ref{g7222} and \ref{g7222qc}.

\begin{figure}[h] 
	\centering
	
	\begin{minipage}{0.8\textwidth}
		\centering
		\includegraphics[width=0.7\linewidth]{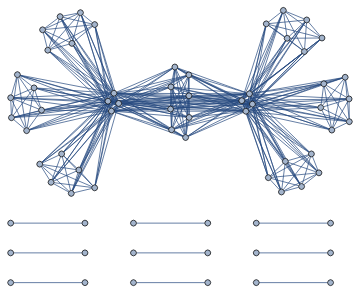}
		\caption{\texttt{$\Gamma (G)$ for $G = $ SmallGroup(72,22)}}
		\label{g7222}
	\end{minipage}
    \bigskip
	\begin{minipage}{0.8\textwidth}
		\centering
		\includegraphics[width=0.7\linewidth]{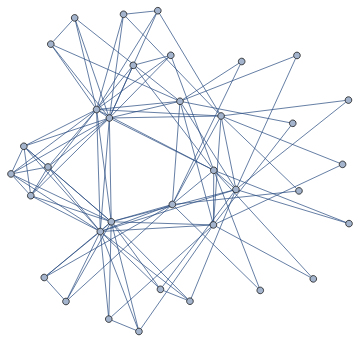}
		\caption{\texttt{$\Gamma (G/Z)$ for $G = $ SmallGroup(72,22)}}
		\label{g7222qc}
	\end{minipage}
\end{figure}

The groups that we have considered all satisfy the condition that $\centerof{G/Z} = 1$.  It makes sense to ask what can be said about groups that satisfy $\centerof{G/Z} = 1$ and do not satisfy $C \cap Z = \{ 1 \}$.  Notice that if $G$ is a such group, then there will definitely exist elements $x,y \in G \setminus Z$ so that $xZ \sim yZ$ in $\Gamma (G/Z)$ but $x \not\sim y$ in $\Gamma (G)$; so the results of Lemma \ref{four} (2) and (3) are not guaranteed to hold.  We will present some examples to show in fact conclusion (3) does not hold.  

First, we claim that it is not hard to see that if $G$ is a Frobenius group with Frobenius kernel $N$, then $\Gamma (G)$ has $1 + |N|$ connected components and if $G$ is a $2$-Frobenius group with normal subgroups $K \le L$ as in the definition above, then $\Gamma (G)$ has $1 + |K|$ connected components.  (Although neither of these computations are difficult, the second computation is done explicitly at the end of Section 3 of \cite{CoLe}.)  Also, Thompson's celebrated theorem shows that a Frobenius kernel is always nilpotent and it known that a Frobenius complement always has a nontrivial center.  Since the nonidentity elements of these subgroups form the connected components of $\Gamma (G)$ when $G$ is a Frobenius group, it follows that each connected component has diameter at most $2$.

When $G$ is a $2$-Frobenius group, there are $|K|$ connected components in $\Gamma (G)$ that each consist of the nonidentity elements of a Frobenius complement of $L$.  Since these Frobenius complements are all cyclic, the corresponding connected components in $\Gamma (G)$ are complete graphs.  The remaining nonidentity elements of $G$ form a single connected component in $\Gamma (G)$.  We know that every element of prime order in $G \setminus L$ centralizes some nonidentity element in $K$.  (See Lemma 3.8 of \cite{CoLe}.)  It follows that every element outside of $K$ has distance at most $2$ to a nonidentity element of $K$, and since $K$ is nilpotent, we deduce that this remaining connected component of $\Gamma (G)$ has diameter at most $6$.  (We would not be surprised if one could actually get a tighter bound for the diameter of this component.)

Take $P$ to be an extra special group of order $p^3$ and exponent $p$ for an odd prime $p$.  Let $q$ be a prime divisor of $p - 1$.  It is well-known that $P$ has an automorphism $\sigma$ of order $q$ that centralizes $\centerof{P}$.  Let $C = \langle \sigma \rangle$ so that $C$ acts on $P$ via automorphisms, and let $G$ be the resulting semi-direct product.  Observe that $\centerof{P} = Z$ and $G/Z$ is a Frobenius group of order $p^2q$.  Also, we have $Z = C (P) \subseteq C (G)$, so we do not have $C \cap Z = \{ 1 \}$.  We see that $\Gamma (G/Z)$ is disconnected and has $1 + p^2$ connected components.  It follows that $\Gamma (G)$ is disconnected.  It is not difficult to see that the $p^2$ connected components of $\Gamma (G/Z)$ that correspond to the Frobenius complements in $G/Z$ will correspond to $p^2$ connected components in $\Gamma (G)$.  

On the other hand, since $Z (P) = Z (G)$, we see that the remaining connected component of $\Gamma (G/Z)$ consists of the nontrivial cosets in $P/Z (P)$.  On the other hand, in the graph of $\Gamma (G)$, the elements in $P \setminus Z$ will have the same connected components as the graph $\Gamma (P)$.  It is not difficult to see that every noncentral element of $P$ has a centralizer that is abelian of order $p^2$.  This implies that $P$ is what is sometimes called a CA-group in the literature.  In particular, it is not difficult to see that the elements in $P \setminus Z$ split into $p + 1$ different connected components each having $p^2 - p$ elements.  In particular, $\Gamma (G)$ has $1 + p + p^2$ different connected components.  This shows that the correspondence in Lemma \ref{four} (3) does not hold when we do not have $C \cap Z = \{ 1 \}$.  

We next present ${\rm GL}_2 (3)$.  In this case, $G/Z \cong S_4$ is a $2$-Frobenius group so that $\Gamma (G/Z)$ has $4$ connected components that are complete graphs and one connected component that has diameter $3$.  On the other hand, $\Gamma (G)$ has $13$ connected components that are all complete graphs.  We include these graphs as Figures \ref{gl23} and \ref{gl23qc}.

\begin{figure}[h] 
	\centering
	
	\begin{minipage}{0.8\textwidth}
		\centering
		\includegraphics[width=0.5\linewidth]{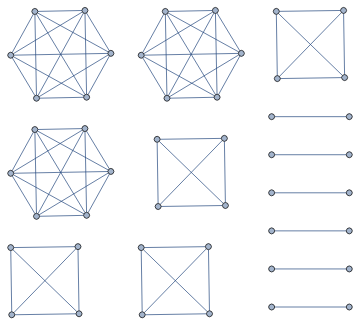}
		\caption{\texttt{$\Gamma (G)$ for $G = \rm{GL}_2 (3)$ }}
		\label{gl23}
	\end{minipage}
\end{figure}
\bigskip
\begin{figure}
	\begin{minipage}{0.8\textwidth}
		\centering
		\includegraphics[width=0.4\linewidth]{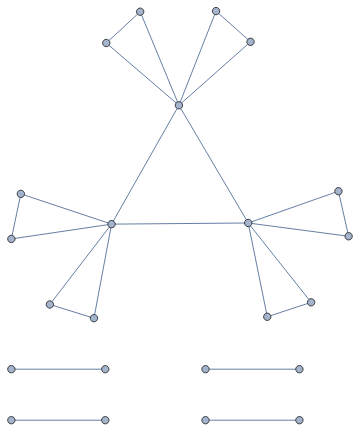}
		\caption{\texttt{$\Gamma (G/Z)$ for $G = \rm{GL}_2 (3)$ }}
		\label{gl23qc}
	\end{minipage}
\end{figure}

\eject
The last examples we present are examples where $\Gamma (G)$ and $\Gamma (G/Z)$ are both connected and the diameters are different.  The groups are SmallGroups (400,125), (400,126), and (400,127).  In all three cases, we have that $\Gamma (G)$ has diameter $5$ and $\Gamma (G/Z)$ has diameter $3$.

\end{document}